\documentclass[12pt,a4paper] {article}
\usepackage{amsmath,amssymb,amsthm}
\usepackage{amsfonts}
\usepackage{mathrsfs}
\usepackage{graphicx}
\usepackage{color}
\usepackage{hyperref}
\usepackage{longtable}
\usepackage{caption2}

\newtheorem{thm}{Theorem} 
\newtheorem{lem}[thm]{Lemma}
\newtheorem{cor}[thm]{Corollary}
\newtheorem{prop}[thm]{Proposition}
\newtheorem{res}[thm]{Result}

\theoremstyle{definition}
\newtheorem{defn}{Definition}

\newtheorem{rem}[thm]{Remark}
\newcommand{\Z}{{\mathbb Z}}
\theoremstyle{remark}

\begin{document}

\title{Nonexistence of Strong External Difference Families in Abelian Groups of Order Being Product of At Most Three Primes}
\author{Ka Hin Leung \footnote{Research is supported by grant
R-146-000-158-112, Ministry of Education, Singapore} \\ Department of Mathematics\\
National University of Singapore\\ Kent Ridge, Singapore 119260\\
Republic of Singapore \\[1cm]
Shuxing Li \footnote{Research is supported by the Alexander von Humboldt Foundation} \\
Faculty of Mathematics\\ Otto von Guericke University Magdeburg, 39106  \\
Magdeburg, Germany \\[1cm]
Theo Fanuela Prabowo \\ Department of Mathematics\\
National University of Singapore\\ Kent Ridge, Singapore 119260\\
Republic of Singapore }

\date{}
\maketitle

\begin{abstract}
Let $v$ be a product of at most three not necessarily distinct primes. We prove that there exists no strong external difference family with more than two subsets in an abelian group $G$ of order $v$, except possibly when $G=C_p^3$ and $p$ is a prime greater than $3 \times 10^{12}$.
\end{abstract}

\section{Introduction}

As an emerging combinatorial configuration proposed in \cite{PS}, strong external difference family (SEDF) has been under intensive study recently \cite{BJWZ,HP,HJN,JL,MS,PS,WYFF}. Roughly speaking, an SEDF is a collection of disjoint subsets of the same size in a group, so that the differences generated from these subsets cover each non-identity element of the group the same number of times. The study of SEDFs is motivated by the so called algebraic manipulation detection (AMD) codes \cite{CDF+}, which can be regarded as a variation of classical authentication codes. Moreover, further cryptographic applications of AMD codes have been discovered later \cite{CFP,CPX}. In \cite{PS}, Paterson and Stinson first described the underlying combinatorial structure behind AMD codes. They proposed the concept of SEDFs, which is equivalent to certain optimal AMD codes. The elegance of this concept can be appreciated from its natural connection to a well-studied configuration named external difference family \cite{CD}, which has applications in synchronization codes \cite{Lev}, as well as authentication codes and secret sharing schemes \cite{OKSS}.

An external difference family may be viewed as some kind of generalization of difference sets. As in the theory of difference sets, we are interested in the construction and nonexistence problems of external difference family.
Character theory and number theory are standard tools in the study of such problems in difference sets. Recently, researchers have successfully employed those techniques in a series of papers \cite{BJWZ,JL,MS,WYFF} to study external difference families that satisfy certain stronger regularity condition.

Let $G$ be a group of order $v$, whose operation is written multiplicatively. Let $D_1, D_2, \cdots, D_m$ be mutually disjoint $k$-subsets of $G$, where $m \geq 2$, and let $\lambda$ be a positive integer. Then $\{D_1, D_2, \cdots, D_m\}$ is a $(v, m, k, \lambda)$-\emph{external difference family} in $G$ if for each nonidentity element $g\in G$, 
$$
\sum_{1\leq i\neq j \leq m} |\{ (x,y) : xy^{-1}=g,  x \in D_j, y\in D_i\}|=\lambda.
$$
Moreover, the collection of subsets $\{D_1, D_2, \cdots, D_m\}$ is a $(v, m, k, \lambda)$-\emph{strong external difference family} (SEDF) in $G$ if for each $1 \leq j \leq m$ and nonidentity $g\in G$,
$$
|\{ (x,y) : xy^{-1}=g,  x \in D_j, y \in \bigcup_{j\neq i=1}^m D_i \}|=\lambda.
$$

Clearly, a $(v,m,k,\lambda)$-SEDF is necessarily a $(v,m,k,m\lambda)$-external difference family.
Note that in any group $G$ of order $v$, there always exists a $(v,v,1,1)$-SEDF by partitioning $G$ into $v$ disjoint subsets, each with one single element. From now on, whenever an SEDF is considered, we always mean \emph{nontrivial} SEDF, i.e., SEDF satisfying $k>1$.

In this paper, we only work on SEDFs over abelian groups, which have been the main focus in this area. For SEDFs in nonabelian groups, we refer to a recent paper \cite{HJN}. In addition, there have been various extensions of SEDFs studied in \cite{HP2,LNC,PS,SM,WYFF}.

As noted in \cite[p. 25]{JL}, there is a fundamental difference between SEDFs with $m=2$ and $m>2$. When $m=2$, there are some few known infinite families, see \cite[Section 4]{BJWZ}, \cite[Theorem 5.6]{HP} and \cite[Example 2.2]{PS}. Some SEDFs with $m=2$ have been constructed, see \cite[Section 3]{DHM} and \cite[Proposition 2.1]{D}, even before the terminology SEDF was adopted.

In contrast, there is only one known nontrivial example if $m>2$. More precisely, there is a $(243,11,22,20)$-SEDF in the group $C_3^5$, which was discovered by two groups of researchers independently and simultaneously \cite[Theorem 3.1]{JL}, \cite[Theorem 3.6]{WYFF}. In the following, we summarize some nonexistence results of SEDF with $m>2$.

\begin{prop}\label{prop-nonexistence}
There is no $(v,m,k,\lambda)$-SEDF in $G$ with $m>2$ if any of following conditions is satisfied:
\begin{itemize}
\item[(a)] $m \in \{3,4\}$ {\rm \cite[Theorems 3.3 and 3.6]{MS}}.
\item[(b)] $\lambda \in \{1,2\}$ {\rm \cite[Corollary 3.2]{HP}, \cite[Theorem 2.2]{PS}}.
\item[(c)] $\lambda \geq k$ {\rm \cite[Lemma 2.4]{HP}}.
\item[(d)] $\lambda>1$ and $\frac{\lambda(k-1)(m-2)}{(\lambda-1)k(m-1)}>1$ {\rm \cite[Theorem 4.1]{HP}}.
\item[(e)] $G=C_p$, where $p$ is a prime {\rm \cite[Theorem 3.9]{MS}}.
\item[(f)] $G=C_{p^2}$, where $p$ is a prime {\rm \cite[Theorem 3.7]{BJWZ}}.
\item[(g)] $k | v$ {\rm \cite[Lemma 1.7]{BJWZ}, \cite[Lemma 1.2]{MS}}
\item[(h)] $\gcd(k,v-1)=1$ {\rm \cite[Lemma 1.5]{JL}}.
\item[(i)] $v-1$ is square-free {\rm \cite[Proposition 2.7]{HJN}}.
\item[(j)] $v$ is a product of distinct primes and $\gcd(mk,v)=1$ {\rm \cite[Corollary 3.6]{BJWZ}}.
\end{itemize}
\end{prop}

Note that Part (g), (h), (i) above are easy consequences of Part (c) and the basic equation \eqref{EqBasic}. Part (j) follows easily from Part (g). Besides those restrictions on the parameters, there is a nonexistence result that concerns with the exponent of the group $G$ \cite[Section 5]{JL}.

So far, we only have one nontrivial example when $m>2$. 
According to \cite[Remark 5.17]{JL}, there is no $(v,m,k,\lambda)$-SEDF with $v \leq 10^5$ and $m \in \{5,6\}$. In addition, except the aforementioned example in $C_3^5$, there are only $70$ plausible parameter sets of $(v,m,k,\lambda)$-SEDF with $v \leq 10^4$ and $m>2$. All these results indicate that nontrivial SEDFs with $m>2$ are very rare.
In this paper, we will further strengthen this belief. Our main results are as follows:

\begin{thm}\label{ThmCyclic}
There exists no SEDF with $m>2$ in any cyclic group of prime power order.
\end{thm}

\begin{thm}\label{ThmMain}
Let $v$ be a product of at most three not necessarily distinct primes. Then there exists no $(v,m,k,\lambda)$-SEDF in $G$ with $m>2$, except possibly when $G=C_p^3$ and $p$ is a prime greater than $3 \times 10^{12}$.
\end{thm}

It turns out that we may view an SEDF as an element in the group ring $\Z[G]$ satisfying certain equation. 
One effective tool to study those equations is character theory. 
In Section \ref{Sec2}, we derive some basic results by applying character theory on SEDFs. For the reader's convenience, we also include some key results in character theory that we employ in later sections.
Using the results obtained in Section \ref{Sec2}, we prove a fundamental inequality between the two parameters $k$ and $\lambda$ in Section \ref{Sec3}. As a consequence, we show that no $(v,m,k,\lambda)$-SEDF exists if $\lambda$ is prime. 
In Section \ref{Sec4}, we study the conditions imposed on the parameters and character values by applying prime power characters. This allows us to prove Theorem \ref{ThmCyclic}.   
In Section 5, we first derive further restrictions by considering characters whose orders are product of two distinct primes instead. Consequently, we are able to deal with cases when the group order has two or three distinct prime divisors. 
In Section \ref{Sec7}, we go back to study the more complicated situation where the group is not cyclic and its order is a prime power. Even though we cannot prove nonexistence of SEDF in general , we are able to prove Theorem \ref{ThmMain}. Finally, we discuss how our main results can be applied to eliminate many unknown cases listed in 
\cite[Remark 5.17]{JL}.

\section{Preliminary results}\label{Sec2}

To define SEDF over an abelian group $G$, it is convenient
to define it over an equation in the group ring $\mathbb{Z}[G]$.
Every element $X\in \Z[G]$ can be written as $X=\sum_{g\in G} a_gg$ with $a_g\in \mathbb{Z}$.
The $a_g$'s are called the \textbf{coefficients} of $X$. For any 
subset $S$ of $G$, we denote the group ring element $\sum_{g\in S}g$ by $S$. 
For any integer $t$ coprime to $|G|$, $X=\sum_{g\in G} a_gg\in \mathbb{Z}[G]$, we define
$X^{(t)}=\sum_{g\in G} a_g g^t$.
Obviously, $\Z[G]$ is a commutative ring with
\[ (\sum_{g \in G} a_gg ) \cdot (\sum_{g \in G} b_gg )= \sum_{g \in G} (\sum_{h \in G}a_{gh^{-1}}b_h) g.\]

\begin{defn}
Suppose $m\geq 2$ and $|G|=v$.
Let $D_1, D_2, \cdots, D_m$ be mutually disjoint $k$-subsets of $G$ and $D=\sum_{i=1}^m D_i$.
$\{D_1, D_2, \cdots, D_m\}$ is called a $(v, m, k, \lambda)$-\emph{external difference family} in $G$ if
 $$
\sum_{\substack{1\le i,j \le m \\ i \ne j}} D_jD_i^{(-1)}=\lambda(G-1_G),
$$
and is called a $(v, m, k, \lambda)$-\emph{strong external difference family} (SEDF) in $G$ if
\begin{equation}\label{eqn-def}
D_j (D^{(-1)}- D_i^{(-1)})=\lambda(G-1_G) \quad \mbox{for each $1 \leq j \leq m$}.
\end{equation}
\end{defn}

As we have mentioned before, we will always assume $G$ is abelian and $m \geq 3$. Throughout this paper,
we assume  $\{D_1,D_2,\ldots, D_m\}$ is a $(v, m, k, \lambda)$-SEDF.

\medskip

To study Equation \eqref{eqn-def}, the standard tools are character theory and elementary number theory.
As we will see later, we often use some divisibility conditions to impose further restrictions on the parameters if $v$ is of certain form. Unfortunately, such arguments are often ad hoc and depend on $v$.

The group of complex characters of $G$ is denoted  by $G^*$.
The \textbf{principal character} of $G$ is the character  $ \chi_0$
with $ \chi_0(g)=1$ for all $g\in G$.
It is well known that $G^*$ is a group isomorphic to $G$,
with multiplication in $G^*$ defined by $ \chi\tau(g)= \chi(g)\tau(g)$
for $ \chi\tau\in G^*$, $g\in G$.
For $X=\sum_{g\in G}a_gg\in \mathbb{Z}[G]$ and $ \chi\in G^*$,
we write $ \chi(X)=\sum_{g\in G} a_g \chi(g)$. Naturally, $\chi$ induces a ring homomorphism from
$\Z[G]$ to $\mathbb{C}$.
 For a more detailed treatment of group rings and characters, please refer to \cite[Chapter VI, Section 3]{BJL}, \cite[Chapter 1]{Pott95} and \cite[Chapter 1]{Sch}

Applying $\chi_0$ on Equation (1), we get
\begin{equation}\label{EqBasic}
(m-1)k^2 = \lambda(v-1).\end{equation}
For $ \chi\neq \chi_0$, $\chi(G)=0$ and we obtain
\begin{equation}\label{EqCharSEDF}
\chi(D_j)(\overline{ \chi(D)} - \overline{ \chi(D_j)}) = - \lambda, \end{equation}
for any $1 \leq j \leq m$.
As shown in \cite{JL}, there are two types of nonprincipal characters with regard to a $(v,m,k,\lambda)$-SEDF.
We define
\begin{align*}
G_0^* &:= \{ \chi \in G^* | \  \chi(D) = 0 \}, \\
G_N^* &:= \{ \chi \in G^* |  \ \chi\neq \chi_0 \mbox{ and } \chi(D) \neq 0   \}.
\end{align*}
The group $G^*$ can be partitioned as disjoint union $G^* = \{ \chi_0\} \cup G_0^* \cup G_N^*$.
First, we record some known results from \cite[Equations (2.2),(2.3)]{JL}.

\begin{res}\label{ResG0}
If $ \chi \in G_0^*$, then $ \chi(D)=0$ and $| \chi(D_i)|^2=\lambda$ for each $1 \le i \le m$.
\end{res}

For convenience, we define the following:

\begin{defn}
For $ \chi \in G^*$, we define
\[ a_{ \chi}:= \min\{| \chi(D_i)|^2\ | 1 \leq i \leq m \} \mbox{ and } \ell_{ \chi}:=|\{ 1\leq i\leq m | \
|\chi(D_i)|^2:=a_{\chi} \}|.\]
We also define $a=\min\{ a_{ \chi} \mid  \chi\in G^*\}$.
\end{defn}

\medskip

We record some known results concerning $a_{\chi}$ and $\ell_{\chi}$.

\begin{res}\label{ResGn}
$G_N^* \neq \emptyset$. For any $ \chi \in G_N^*$, $a_{\chi}$ is an integer, $a_{\chi}<\lambda$;  and for $1\leq i\leq m$,
\[ | \chi(D_i)|^2= a_{ \chi} \mbox{ or }  \ \frac{\lambda^2}{a_{ \chi}}. \]
Furthermore,   the following holds: 
\begin{itemize}
\item [(a)] $a_{ \chi} | \lambda^2$.
\item [(b)] $a=a_{ \chi}$ for some $ \chi\in G_N^*$.
\item [(c)] $ m/2 < \ell_{ \chi}\leq m-2$.
\item [(d)] Let $d=\gcd(\lambda, a_{ \chi})$. Then $a_{ \chi}| d^2$. 
\end{itemize}
\end{res}
\begin{proof} The first statement  follows easily from \cite[Lemma 3.1(d)]{MS}. 
In view of \cite[Equation (5.3)]{JL} and the equation after that, we see that 
$|\chi(D_i)|^2=a_{\chi}$ or $\lambda^2/a_{\chi}$. 
As shown in the proof of \cite[Lemma 5.1]{JL}, both $a_{\chi}$ and $\lambda^2/a_{\chi}$ are integers. 
Part (a) is now clear. Since $G_N^*$ is nonempty and $a_\chi<\lambda$ for any $\chi\in G_N^*$, Part (b) follows. 
Part (c) is an easy consequence of \cite[Equation (5.2)]{JL}. 
To show (d), write 
\[ \lambda=\prod q_i^{t_i} \mbox{ and } a_{ \chi}=\prod q_i^{s_i} \]
where $q_i$'s are distinct primes. As $a_{ \chi}|\lambda^2$, we see that $2t_i\geq s_i$. Clearly,
$d=\prod q_i^{\min(t_i, s_i)}$. Therefore, $d^2=\prod q_i^{2\cdot \min(t_i,s_i)}$. Obviously,
$2\cdot \min(t_i, s_i)\geq s_i$ as $s_i\leq 2t_i$. Thus $a_{\chi}|d^2$. 
\end{proof}

%

\medskip

We next prove a simple but crucial observation.

\begin{lem}\label{LemCharBas}
If $ \chi \in G_N^*$ and $ \chi(D_i) \neq  \chi(D_j)$, then $ \chi(D-D_i-D_j) = 0$ and $ \chi(D_iD_j^{(-1)})=-\lambda$.
\end{lem}

\begin{proof}
By Result \ref{ResGn} (c), there exist $1 \le i, j \le m$, such that $ \chi(D_i) \ne  \chi(D_j)$. Let $X:= D-D_i-D_j$. Equation \eqref{eqn-def} gives
\[ D_i(X^{(-1)} +D_j^{(-1)}) = \lambda(G-1_G) \mbox{ and }D_j^{(-1)}(X+D_i) = \lambda(G-1_G). \]
Since $G$ is abelian, we conclude $D_iX^{(-1)} = D_j^{(-1)}X$. 
Therefore, 
\[ | \chi(D_i)|^2 | \chi(X)|^2 = | \chi(D_j)|^2 | \chi(X)|^2.\]
Since $ \chi(D_i) \neq  \chi(D_j)$, $ |\chi(D_i)|^2 \neq  |\chi(D_j)|^2$  by  \cite[Equation (5.3)]{JL}.
Hence, $ \chi(X) = 0$.
Lastly, applying $ \chi$ to $D_i(X^{(-1)} + D_j^{(-1)}) = \lambda(G-1_G)$ gives $ \chi(D_iD_j^{(-1)})=-\lambda$.
\end{proof}

The above result can be extended to two characters in $G_N^*$.

\begin{cor}\label{LemSeparate}
If $ \chi_1, \chi_2 \in G_N^*$, then there exist distinct $1 \le i,j \le m$ such that $ \chi_1(D-D_i-D_j) =  \chi_2(D-D_i-D_j)=0$.
\end{cor}
\begin{proof}
For $t \in \{1,2\}$, let $S_t:= \{1 \leq j \leq m \mid | \chi_t(D_j)|^2 = a_{ \chi_t} \}$.
By Lemma \ref{LemCharBas}, it suffices to show that 
there exist distinct $i,j$ such that $\chi_t(D_i)\neq \chi_t(D_j)$ for $t=1,2$. 

Suppose $S_1 \subset S_2$ or $S_2 \subset S_1$. Without loss of generality, we may assume
$S_1 \subset S_2$. We pick $i\in S_1\cap S_2$ and $j\notin S_2$. Then
$\chi_t(D_i)\neq \chi_t(D_j)$ for $t=1,2$. We may now assume both $S_1\backslash  S_2$ and $S_2\backslash  S_1$ are nonempty. Then, there exist  $i\in S_1\backslash  S_2$ and $j\in S_2\backslash S_1$. Again,
$\chi_1(D_i)\neq \chi_1(D_j)$ and $\chi_2(D_i)\neq \chi_2(D_j)$. 
\end{proof}

\medskip

Lemma \ref{LemCharBas} also allows us to find some useful relations between $\ell_{\chi}$ and $m$.

\begin{cor}\label{CorBasic}
Let $ \chi \in G_N^*$. Then
\begin{itemize}
\item [(a)]  $ (\ell_{ \chi}-1)(\lambda+a_{ \chi})=\lambda(m-2)$.
\item [(b)]  $(m - \ell_{ \chi}-1)(\lambda+a_{ \chi})=a_{ \chi}(m-2)$,
\item [(c)]  $(\lambda+a_{ \chi})$ divides $(m-2)\cdot \gcd(\lambda,a_{ \chi})$.
\end{itemize}
\end{cor}

\begin{proof}
By Result \ref{ResGn} and Lemma \ref{LemCharBas}, there exist $1 \leq i,j \leq m$ such that
\[ | \chi(D_i)|^2 = a_{ \chi},  \ | \chi(D_j)|^2 = \frac{\lambda^2}{a_{ \chi}} \mbox{ and } \chi(D-D_i-D_j) = 0.\] Consequently, $(\ell_{ \chi}-1) \chi(D_i) + (m - \ell_{ \chi}-1) \chi(D_j) = 0$.
Hence,
\[ \frac{(\ell_{ \chi}-1)^2}{(m - \ell_{ \chi}-1)^2} = \frac{| \chi(D_j)|^2}{| \chi(D_i)|^2} = \left(\frac{\lambda}{a_{ \chi}} \right)^2.\]
This shows
$$\frac{\ell_{ \chi}-1}{m - \ell_{ \chi}-1} = \frac{\lambda}{a_{ \chi}}.$$
Therefore,
\[ \ell_{ \chi}-1=\frac{\lambda(m-2)}{\lambda+a_{ \chi}} \ \mbox{ and } \ m - \ell_{ \chi}-1=\frac{a_ \chi(m-2)}{\lambda+a_{ \chi}}.\]
Thus $(\lambda+a_\chi)$ divides $a_\chi(m-2)$. 
As $\gcd(\lambda, a_ \chi)=\gcd(\lambda+a_ \chi, a_ \chi)$; and 
$(\lambda+a_{\chi})/\gcd(\lambda, a_ \chi), a_ \chi/ \gcd(\lambda, a_{ \chi})$ are relatively prime, (c) follows.
\end{proof}

\medskip

Next, we record some useful results in character theory that will be employed later.
The following is a standard result, see \cite[Chapter VI, Lemma 3.5]{BJL}, for instance.

\begin{res}[Fourier inversion formula] \label{fourier}
Let $G$ be a finite abelian group and $X=\sum_{g\in G}a_gg \in\mathbb{Z}[G]$. Then
$$a_g=\frac{1}{|G|}\sum_{ \chi\in G^*} \chi(Xg^{-1})  \text{ for all } g\in G.$$
\end{res}

\medskip

It is then immediate to deduce the following:

\begin{res}
If $A\in \Z[G]$ and $\chi(A)=0$ for all nonprincipal $\chi\in G^*$, then $A=xG$ for some integer $x$.
\end{res}

Finally, we turn our attention to the structure of $D_iD_i^{(-1)}$. Observe that in view of Results \ref{ResG0} and \ref{ResGn},
$\chi(D_iD_i^{(-1)})\in \Z$ for any $\chi\in G^*$.  This observation is crucial in knowing the structure of $D_iD_i^{(-1)}$. We first define
\[ \omega(g)=\{g^t \ | \ t \in \mathbb{Z}, \gcd(t,|G|) = 1\}\]
for any $g\in G$. We call $\omega(g)$ the orbit of $g$.

\begin{lem}\label{LemOrbit0}
Let $A \in \mathbb{Z}[G]$. If $ \chi(A) \in \mathbb{Q}$ for any $ \chi \in G^*$, then
\[ A=\sum_{i=1}^r a_i\omega(g_i) \]
for some $g_1,\ldots, g_r \in G$ and $a_1,\ldots, a_r\in \mathbb{Z}$.
\end{lem}
\begin{proof}
Write $A=\sum a_g g\in \mathbb{Z}[G]$, to show $A$ is of desired form, it suffices to show that $a_{g}=a_{g^t}$ for any $t$ coprime with $|G|$.
By Result \ref{fourier},
$$a_g=\frac{1}{|G|}\sum_{ \chi\in G^*} \chi(Ag^{-1})  \mbox{ and } a_{g^t}=\frac{1}{|G|}\sum_{ \chi\in G^*} \chi(Ag^{-t})  $$
Let $v=|G|$. For any $t$ with $\gcd(t,v) = 1$, we define
$\sigma_t \in {\rm Gal}(\mathbb{Q}(\zeta_v)/\mathbb{Q})$ such that  $\sigma_t(\zeta_{v})=\zeta_{v}^t$.
Note that
\[  \sigma_t (\chi(Ag^{-1}))= \sigma_t(\chi(A))\sigma_t(\chi(g^{-1}))=\chi(A)\chi(g^{-t})\]
as $\chi(A)\in \mathbb{Q}$.
Therefore,
\[ a_g=\sigma_t(a_g)=\frac{1}{|G|}\sum_{ \chi\in G^*} \sigma_t(\chi(Ag^{-1}) ) =
\frac{1}{|G|}\sum_{ \chi\in G^*} \chi(Ag^{-t})=a_{g^t}.\]
\end{proof}

\medskip

We record the next two results which follow from Results \ref{ResG0}, \ref{ResGn} and Lemmas \ref{LemCharBas}, \ref{LemOrbit0}.

\begin{cor}\label{CorDiDi}
For any $1 \le i \le m$, there exist $g_j$'s in $G$ and $a_j$'s in $\Z$ such that
\[ D_iD_i^{(-1)}=\sum_{j=1}^{r} a_j\omega(g_j) .\]
\end{cor}

\begin{cor}\label{CorDiDj}
Suppose $G_0^*=\emptyset$. For any $1\leq i,j\leq m$, there exist $g_t$'s in $G$ and $a_t$'s in $\Z$ such that
\[ D_iD_j^{(-1)}=\sum_{t=1}^{r} a_t\omega(g_t) \]
\end{cor}

To conclude this section, we remind the reader of the following result on character theory that we use often. 

\begin{res}
Let $\Omega$ be a subgroup of $G^*$ and $N=\bigcap_{\chi\in \Omega} \ker (\chi)$.
For any $\chi \in \Omega$, $\chi$ induced a character on $G/N$ also denoted by $\chi$ such that  
\[ \chi(gN)=\chi(g) \mbox{ for all } g\in G.\]
Then $\Omega$ can be regarded as $(G/N)^*$. Moreover, if $\pi: G\rightarrow G/N $ is the natural projection, then for any $A\in \Z[G]$, $\chi(A)=\chi (\pi(A))$ for any $\chi\in \Omega$.
\end{res}

\section{A fundamental result}\label{Sec3}

In this section, we are going to derive some fundamental relations among the parameters of an SEDF, which will be frequently used later. The next theorem gives a very important bound on $k$. Surprisingly, the result is not recorded earlier.

For convenience, we keep the notations used in Section \ref{Sec2}. Recall that for any $ \chi \in G^*$, we define
\[ a_{ \chi}:= \min\{| \chi(D_i)|^2\ | 1 \leq i \leq m \} \mbox{ and } a:=\min\{ a_{ \chi} \ |\   \chi\in G^*\}.\]

\begin{thm}\label{ThmBound}
$$ \lambda < k < \frac{v-1}{v} \lambda + \frac{\lambda^2}{ma}. $$
Furthermore, if we set $d=\gcd(\lambda, a)$ and $x=(m-2)d/(\lambda+a)$.  Then,
\[ \lambda < k < (1+\frac{d}{ax}) \lambda < 2\lambda . \]
\end{thm}

\begin{proof}
The lower bound on $k$ follows from Proposition~\ref{prop-nonexistence} (c). Consider the coefficient of $1_G$ in $\sum_{i=1}^m D_iD_i^{(-1)}$. By Fourier inversion formula, we have
\[ mk= \frac{1}{v} ( mk^2 + \sum_{ \chi \in G_0^*} \sum_{i=1}^m | \chi(D_i)|^2 + \sum_{ \chi \in G_N^*} \sum_{i=1}^m | \chi(D_i)|^2 ).\]
For any $ \chi\in G_N^*$, it follows from Corollary \ref{CorBasic}  that
\[ (\ell_{ \chi}-1)a_{ \chi}+(m-\ell_{ \chi}-1)\frac{\lambda^2}{a_{ \chi}}=
(m-2) \frac{\lambda }{\lambda+a_{ \chi}}a_{ \chi}+(m-2) \frac{a_{ \chi}}{\lambda+a_{ \chi}}
\frac{\lambda^2}{a_{ \chi}}=(m-2) \lambda.\] Hence,
\begin{align*}
 \sum_{i=1}^m | \chi(D_i)|^2  & =  (m - \ell_{ \chi}-1) \frac{\lambda^2}{a_{ \chi}} + (\ell_{ \chi}-1) a_{ \chi} + (\frac{\lambda^2}{a_{ \chi}} +a_{ \chi}) \\
 & = (m-2)\lambda+ (\frac{\lambda^2}{a_{ \chi}} +a_{ \chi})=m\lambda + \frac{(\lambda-a_{ \chi})^2}{a_{ \chi}} .\end{align*}
On the other  hand, if $ \chi\in G_0^*$, then $a_{ \chi}=\lambda$, hence
\[ \sum_{i=1}^m | \chi(D_i)|^2 =  m \lambda + \frac{(\lambda-a_{ \chi})^2}{a_{ \chi}}.\]
As $a\leq a_{ \chi}$ for all nontrivial characters $ \chi$, we conclude
\begin{equation}\label{EqInequ}
k = \frac{1}{v} [k^2 + (v-1)\lambda+ \frac{1}{m}\sum_{\chi_0\neq  \chi \in G^*} \frac{(\lambda -a_{ \chi})^2}{a_{ \chi}}] < \frac{\lambda}{m-1} + \frac{v-1}{v} \lambda + \frac{(\lambda-a)^2}{ma}. \end{equation}
To simplify further, we get
\[ k< \frac{v-1}{v} \lambda + \frac{\lambda }{m-1} - \frac{2\lambda }{m}+  \frac{\lambda^2+a^2}{ma}
= \frac{v-1}{v} \lambda + \frac{\lambda^2}{ma} -[ \frac{(m-2)\lambda}{m(m-1)} -  \frac{a}{m}].\]
In view of Corollary \ref{CorBasic} (a) and (b), we see that $\lambda/a\geq (m-1)/(m-3)$. Thus
\[ \frac{(m-2)\lambda}{m(m-1)} -  \frac{a}{m} >0.\]
Therefore,
\[ k < \frac{v-1}{v} \lambda +  \frac{\lambda^2}{ma}<(1+\frac{d}{ax})\lambda \leq  2\lambda\]
as by assumption $m-2=x(\lambda+a)/d$.  \end{proof}

\medskip

Theorem \ref{ThmBound} gives a useful bound relating $k$ and $\lambda$. 
Recall that in \cite[Remark 5.17]{JL},  there are only $70$ plausible parameter sets of $(v,m,k,\lambda)$-SEDFs with $v \leq 10^4$ and $m>2$. We will illustrate in the last section how Theorem \ref{ThmBound} can be applied to rule out 44 cases listed there. 
We believe the above theorem is crucial in proving nonexistence results on SEDFs.
For example, \cite[Corollary 5.6]{JL} is an easy consequence of Theorem \ref{ThmBound}. As another illustration, we record an immediate consequence:

\begin{cor}\label{CorLambdaPrime}
There exists no $(v,m,k,\lambda)$-SEDF with $m>2$ when $\lambda$ is a prime. 
\end{cor}
\begin{proof} 
In view of Equation \eqref{EqBasic} and the assumption that $\lambda$ is a prime, we
see that either
\[ \lambda|k \mbox{ or } \lambda|(m-1). \]
If $\lambda|k$, then as $k>\lambda$, $k\geq 2\lambda$. This contradicts
Theorem \ref{ThmBound}. Thus, $\lambda|(m-1)$.

Let $a$ be as defined before. By Result \ref{ResGn}, $a < \lambda$ and $a|\lambda^2$.
So $a=1$. By Corollary \ref{CorBasic}, $m-2 =s (\lambda + 1)$ for some integer $s$.
Recall that $\lambda|(m-1)$, hence  $-1 \equiv s \bmod \lambda$. So, $s=(t\lambda-1)$ and
\[ m-2\geq s(t\lambda-1)(\lambda+1)\geq \lambda^2-1.\]
Consequently, $m \geq \lambda^2 + 1$. Theorem \ref{ThmBound} implies that 
\[ \lambda< k < \frac{v-1}{v} \lambda + \frac{\lambda^2}{\lambda^2+1}<\lambda+1.\]
This is impossible. 
\end{proof}

\begin{rem}
It follows from the above corollary and Proposition 1 (b)  that $\lambda\geq 4$ and $k\geq 5$. From now on, we may assume $m\geq 5$, $k\geq 5$ and $\lambda\geq 4$. 
\end{rem}

\section{Characters of prime power order}\label{Sec4}

We will continue using the notations defined before and assume $\{D_1,\ldots, D_m\}$ is a $(v,m,k,\lambda)$-SEDF on $G$. To understand the relation between the prime divisors of the group order $v$ and other parameters of an SEDF, we study characters of prime power order in $G^*$.

\begin{lem}\label{LemPpower}
Let $p$ be a prime and $H$ be a cyclic subgroup with $|H|=p^t$. Suppose $\chi$ is a character of order $p^t$.
\begin{itemize}
\item [(a)] If $A=\sum a_g g\in \Z[H]$ and $\chi(A)=0$, then $A=PX$ for some $X\in \Z[G]$. In particular, 
$p|(\sum a_g)$. 
\item [(b)] If $Y=\{0\leq i\leq t-1: \chi^{p^i}(A)=0\}$, then $p^{|Y|}|(\sum a_g)$.
\end{itemize}
\end{lem}
\begin{proof} (a) is known. A proof can be found in \cite[Theorem 3.3]{LL}. 
To prove (b), we apply induction on $|Y|$. We are done if $|Y|=0$. 
We may assume $|Y|\geq 1$. 

Let $s=\min \{i: i\in Y\}$ and $N=\ker(\langle \chi^{p^{s}} \rangle)$. 
Let $\pi:H\rightarrow H/N$ be the natural projection. 
We may identify $(H/N)^*=\langle \chi^{p^s} \rangle$. 
Since $\chi^{p^s}(A)=0$, it follows from (a) that $\pi(A)=P Z$ where $P$ is the cyclic subgroup of order $p$ in $H/N$ and $X\in \Z[H/N]$. Write $Z=\sum b_z z\in \Z[H/N]$. 
Clearly, $\sum a_g=p \sum b_z$. 
For any $i\in Y$ with $i>s$, $\chi^{p^i}(A)=p\chi^{p^i}(Z)$. Hence, 
\[ Y'=\{ 0\leq i \leq t-s-1: \chi^{p^{s+i}}(Z)=0\}=\{ i-s: i >s  \mbox{ and } i \in Y\}. \]
Note that $|Y'|=|Y|-1$. Hence by induction, $p^{|Y'|} | (\sum b_z)$. Since $|Y'|=|Y|-1$, we conclude $p^{|Y|}| (\sum a_g)$. 
\end{proof}

\medskip

We now consider characters of prime power order in $G^*$. When the character is in $G_0^*$, we have the following immediate consequence of Lemma~\ref{LemPpower}.

\begin{lem}\label{LemDivChar1}
Suppose $ \chi\in G_0^*$ has order $p^{\alpha}$ for some prime $p$.
Let  $Y=\{ 0\leq t\leq \alpha-1 \ |\ \chi^{p^t}\in G_0^*\}$. Then $|Y|\geq 1$ and $p^{|Y|}|(mk)$.
\end{lem}

When the character is in $G_N^*$, we have more detailed information.

\begin{lem}\label{LemCharPrime}
Let $ \chi$ be a character of prime power order $p^{\alpha}$ in $ G_N^*$. Then, we have the following:
\begin{itemize}
\item [(a)] $p| (m-2)k$.
\item [(b)] $p|(k^2 + \lambda)$, $p|(k^2 - a_{ \chi})$ and $p|(\lambda + a_{ \chi})$.
\item [(c)] If $p\nmid k$, then $p \nmid \lambda$, $p | [(\lambda + a_{ \chi})/\gcd(\lambda, a_{ \chi})]$ and $p\leq \lambda+1$. 
\item [(d)] If $\chi^{p^t}\in  G_N^*$ for some $t>0$, then $p^2| (m-2)k$.
\end{itemize}
\end{lem}
\begin{proof}
Let $H=\ker( \langle \chi\rangle)$. Then $ \chi$ induces a character of order $p^{\alpha}$ on $G/H$. By abuse of notation, we denote the induced character by $ \chi$ as well. Let $\pi: G \rightarrow G/H$ be the natural projection.
By Result \ref{ResGn}, there exist $1 \leq i,j \leq m$ such that
\[ | \chi(D_i)|^2 = a_{ \chi} \ \mbox{ and } \  | \chi(D_j)|^2 = \frac{\lambda^2}{a_{ \chi}}.\]
By Lemma \ref{LemCharBas}, $ \chi(\pi(D-D_i-D_j))=0$. By Lemma \ref{LemPpower}, $p|[(m-2)k]$ as $\chi_0(D-D_i-D_j)=(m-2)k$.

For (b), note that $ \chi(\pi(D_i D_i^{(-1)})-a_{ \chi}1_G)=0$. That means
$\pi(D_i D_i^{(-1)}-a_{ \chi}1_G) =PY$ where $Y \in \mathbb{Z}[G/H]$. By applying the principal character on the equation, we get $p| (k^2-a_{ \chi})$. 

If $p|k$, then $p|a_{ \chi}$. But as $a_{ \chi}|\lambda^2$, so $p|\lambda$ as well. Therefore, (b) holds if $p|k$. We may assume $p\nmid k$. By (a), $p|(m-2)$ and therefore, $p\nmid (m-1)$.
By Equation \eqref{EqBasic}, $p \nmid \lambda$.
Using the same equation, we see that $k^2\equiv -\lambda \bmod p$ as $p|v$. Thus $p|(k^2+\lambda)$.
Recall that $p|(k^2-a_{ \chi})$, we obtain $p|(\lambda+a_{ \chi})$. 
Let $d=\gcd(\lambda, a_{\chi})$. If $d\geq 2$, then $dp| (\lambda +a_{\chi})$ and 
$2\lambda > dp$. Thus $p< \lambda$. If $d=1$, then by Result 5 (a), $a_\chi=1$. Hence, $p\leq (\lambda+1)$. This finishes proving (b) and (c). 

For (d), we conclude from Corollary \ref{LemSeparate} that there exist $i,j$ such that
\[ \chi(\pi(D-D_i-D_j))=\chi^{p^t}(\pi(D-D_i-D_j))=0.\]
Clearly, (d) follows from Lemma \ref{LemPpower}.
\end{proof}

\medskip

Combining Lemmas~\ref{LemDivChar1} and \ref{LemCharPrime}, we know that a character $\chi$ of order $p$ is in $G_0^*$ only if $p \mid mk$ and in $G_N^*$ only if $p \mid (m-2)k$. Therefore, for $p$ not dividing $k$, we have either $p \mid m$ or $p \mid (m-2)$. Indeed, we can derive stronger divisibility result as follows.

\begin{cor}\label{CorG0}
Suppose ${\cal P}^*$ is the Sylow $p$-subgroup of $G^*$. If $p$ is odd, $p\nmid k$ and $p|m$, then ${\cal P}^* \subset G_0^* \cup \{\chi_0\}$ and $|{\cal P}^*|$ divides $m$.
\end{cor}
\begin{proof} Since $p|m$ and $p$ is odd, $p\nmid (m-2)$. As $p\nmid k$ also, thus $p\nmid (m-2)k$. 
In view of Lemma \ref{LemCharPrime} (a), we conclude that for
any nonprincipal $ \chi \in {\cal P}^*$, $\chi\in G_0^*$. We define 
\[ N=\bigcap_{ \chi\in {\cal P}^*} \ker( \chi). \]
Let $\pi: G\rightarrow G/N$ be the natural projection. Then ${\cal P}^*$ may be viewed as $(G/N)^*$.
As ${\cal P}^*\subset G_0^*\cup \{\chi_0\}$, $ \chi(\pi(D))=0$ for all nonprincipal $ \chi\in (G/N)^*$. Therefore, by Result 10, $\pi(D)=x(G/N)$ and $|{\cal P}^*|$ divides $m$.
\end{proof}

\medskip

It does not seem possible to extend Corollary \ref{CorG0} to the case $p|(m-2)$. We only manage to prove a much weaker result. Before proving that, we need a lemma.

\begin{lem}\label{LemPower}
Let $N$ be a subgroup of $G$ such that $G/N$ is a cyclic subgroup of order $p^{\alpha}$.
Suppose $p$ is a prime and $\pi: G \rightarrow G/N$ is the natural projection.
If $ \chi\in G_N^*$ is of order $p$ and $N \subset \ker ( \chi)$, then
$p^{\alpha}|(\lambda^2-a_{ \chi}^2)$.
\end{lem}
\begin{proof}
Let $g$ be a generator of $G/N$.
There exist $D_i$ and $D_j$ such that $| \chi(D_i)|^2=a_{ \chi}$ and
$| \chi (D_j)|^2=\lambda^2/a_{ \chi}$.
Let $t_g$ be the coefficient of $g$ in
$\pi(D_iD_i^{(-1)}-D_jD_j^{(-1)})$.
By Fourier inversion formula,
\[ p^{\alpha} t_g= \sum_{\phi\in (G/N)^*} \phi(\pi(D_iD_i^{(-1)}-D_jD_j^{(-1)})) \phi(g^{-1}).\]
Let $ \chi=\chi_1, \chi_2,\ldots,  \chi_{\alpha}$ be characters of order $p, p^2,\ldots, p^{\alpha}$ in $(G/N)^*$.
Let $\Omega(i)=\{  \chi_i^t: p\nmid t\}$. Note that $\phi(\pi(D_iD_i^{(-1)}-D_jD_j^{(-1)})) \in \mathbb{Z}$ for all $\phi\in (G/N)^*$. Hence,
\[ \sum_{\phi \in \Omega(i)} \phi(\pi(D_iD_i^{(-1)}-D_jD_j^{(-1)})) \phi(g^{-1})=
\chi_i(\pi(D_iD_i^{(-1)}-D_jD_j^{(-1)}))  \sum_{\phi \in \Omega(i)} \phi(g^{-1}).\]
Note that
\[  \sum_{\phi \in \Omega(i)} \phi(g^{-1})=\left\{\begin{array}{ll}
-1 & \mbox{if } i=1\\
0 & \mbox{if } 2 \le i \le \alpha .
 \end{array}\right.\]
Thus
$p^\alpha t_g = \lambda^2/a_{ \chi}-a_{\chi}=(\lambda^2-a_{\chi}^2)/a_{\chi}$.
\end{proof}

\begin{rem}
When $p$ is odd, the result in the above lemma can be further strengthened. If
$p$ is odd and $p\nmid k$, then $p^{\alpha}|(\lambda+a_{ \chi})$
as by Lemma~\ref{LemCharPrime} (b), $p|(\lambda+a_{ \chi})$ and $p\nmid (\lambda-a_{ \chi})$.
\end{rem}

\begin{cor}\label{LemDivChar3}
Suppose $p$ is a prime and there exists a character $ \chi \in G_N^*$ of order $p^{\alpha}$ for some $\alpha \geq 1$. If $p \nmid k$ and $p$ is odd, then $p^{\alpha}|(m-2)$.
\end{cor}
\begin{proof}
By assumption, $|G/\ker ( \chi)|=p^{\alpha}$. To apply the previous lemma, we first show
 $ \chi^{p^{\alpha-1}}\in G_N^*$. As $p\nmid k$, it follows from Lemma \ref{LemCharPrime} that $p|(m-2)$. Since $\psi=\chi^{p^{\alpha-1}}$ is a character of order $p$ in $G_N^*$, it follows from Lemma \ref{LemPower} and Remark 23 that $p^{\alpha}|(\lambda+a_{\psi})$.
Our result follows now from Corollary \ref{CorBasic} (c).
\end{proof}

\medskip

It is not clear whether $|{\cal P}^*|$ divides $(m-2)$ even if ${\cal P}^*\subset G_N^*\cup \{\chi_0\}$. However, we have developed enough tools to prove Theorem \ref{ThmCyclic}.

\begin{proof}[Proof of Theorem \ref{ThmCyclic}]
Suppose $v=p^{\alpha}$, where $p$ is a prime. For any $1\leq t\leq \alpha$, we let $\chi_t$ be a character of order $p^t$.

Suppose $ \chi_{\alpha} \in G_N^*$. Since $m>2$, there exist $1 \leq i, j \leq m$ such that $ \chi_{\alpha}(D_i) =  \chi_{\alpha}(D_j)$, i.e. $ \chi_{\alpha}(D_i-D_j)=0$. Then by Lemma \ref{LemPpower} (a), $D_i-D_j = PX$ for some $X \in \mathbb{Z}[G]$, where $P$ is the unique subgroup of order $p$ in $G$. As $D_i$ and $D_j$ are disjoint subsets in $G$,
$D_i$ must be a union of $P$-cosets. So, $ \chi_{\alpha}(D_i) = 0$, contradicting Result \ref{ResGn}. Thus, $ \chi_{\alpha} \in G_0^*$ and so $p|mk$.
On the other hand, as $G_N^* \neq \emptyset$, so by Lemma \ref{LemCharPrime} (a), $p|(m-2)k$. Consequently,
$p|(2k)$.

We claim $p\nmid k$. Otherwise, it is then clear that $p^2 | \lambda$. As $ \chi_{\alpha} \in G_0^*$,
$| \chi_{\alpha}(D_i)|^2 = \lambda$ is divisible by $p^2$. As prime ideals above $p$ are invariant under complex conjugation, we see that $p| \chi_{\alpha}(D_i)$. By Ma's Lemma \cite[Lemma 1.5.1]{Sch}, $D_i = pX_0 + PX_1$ for some $X_0, X_1 \in \mathbb{Z}[G]$. As all nonzero coefficients in $D_i$ are $1$, this is impossible unless
$X_0=0$. That implies $D_i = PX_1$ and therefore $ \chi_{\alpha}(D_i)=0$, contradicting Result \ref{ResG0}.

We may assume $p\nmid k$. Recall that $p|(2k)$. That means $p=2$, $k$ is odd and $m$ is even. In particular, $(m-1)$ and $\lambda$ are odd. Let $Z=\{ 1\leq i\leq \alpha \mid \chi_i \in G_N^*\}$. 
Recall that we have $|Z|>0$.  

Case (1): $|Z|=1$. Then, by Lemma \ref{LemDivChar1}, $2^{\alpha-1}|(mk)$. That means $2^{\alpha-1}|m$. By Remark 17,  $k\geq 5$. Thus $mk>2^{\alpha}$. This is impossible.

Case (2): $|Z|\geq 2$. Then by Lemma~\ref{LemCharPrime} (d), $4|(m-2)k$. Since $k$ is odd, $4\nmid (mk)$. By Lemma~\ref{LemDivChar1}, $|Z|=\alpha -1 $. As $\chi_{\alpha}\in G_0^*$,  $\chi_1\in G_N^*$.
By Lemma~\ref{LemPower}, $2^{\alpha}|(\lambda^2-a_1^2)$. Since $\lambda$ is odd, $2^{\alpha-1}|(\lambda-a_1)$ or
$2^{\alpha-1}|(\lambda+a_1)$. If $2^{\alpha-1}|(\lambda-a_1)$, then $\lambda\geq 2^{\alpha-1}$.  If $2^{\alpha-1}|(\lambda+a_1)$, then $\lambda\geq 2^{\alpha-2}$. In any case, $4\lambda \ge v$. But by Proposition 1 (a), $m\geq 5$, so $v> mk > m \lambda > 4 \lambda=2^{\alpha}$. This is impossible.
\end{proof}

\section{$v=pqr$}\label{Sec5}

We keep our notations used earlier. Again, we assume $\{D_1,\ldots, D_m\}$ is a $(v,m,k,\lambda)$-SEDF on $G$. For any divisor $r$ of $v$,  we assume $ \chi_r\in G^*$  of order $r$ and for simplicity, we denote $a_{ \chi_r}$ by $a_r$. In the previous section, we only deal with characters of prime power order. 
To deal with the case $v$ has at least two distinct prime divisors $p,q$, it is then natural to consider characters of order $pq$.

\begin{lem}\label{LemCharCompare}
 Let $ \chi_p, \chi_q, \chi_{pq}, a_p, a_q$ and $a_{pq}$ be as defined above. 
Suppose $\chi_{pq}=\chi_p\chi_q$. Then the following holds.
\begin{itemize}
\item [(a)] $q| (a_{p} - a_{pq})$.
\item [(b)] If $ \chi_p \in G_0^*$ and $ \chi_{pq} \in G_N^*$, then
$q|(\lambda - a_{pq})$. In particular, $q<\lambda$.
\item [(c)] If $ \chi_{pq} \in G_0^*$ and $ \chi_p \in G_N^*$,
then $q|(\lambda - a_{p})$. In particular, $q<\lambda$.
\item [(d)] If $ \chi_p\in G_0^*$, $ \chi_q ,  \chi_{pq} \in G_N^*$, $q\nmid k$ and $q$ is odd, then $\lambda > \max(p,q)$ and $m \ge q(q+2)+2$. Moreover,
\[ v> m\lambda > \max(pq^2, q^3).\]
\item [(e)] If $ \chi_p,  \chi_{pq} \in G_0^*$, $ \chi_q \in G_N^*$,
 $\gcd(pq,k)=1$ and $p$ is odd, then $\lambda\geq \max(p+1,q-1)$ and there exists a positive integer $x$ such that
\[ m=(px-1)\left(\frac{\lambda}{a_q}+1\right)+2.\]
\end{itemize}
\end{lem}
\begin{proof}
Let $H=\ker( \chi_{pq})$. Since $ \chi_{pq}$ has order $pq$, then $G/H$ is cyclic and of order  $pq$.
As before, we may view $\chi_p, \chi_q$ and $\chi_{pq}$ as characters in $(G/H)^*$.
Let $\pi: G \rightarrow G/H$ be the natural projection.
$G/H$ has four orbits $\{1_{G/H}\}, P', Q'$ and $(PQ)'$ in $G/H$, consisting of all elements of order $1$, $p$, $q$ and $pq$ respectively. Note that $ \chi_p(\pi(D_iD_i^{(-1)})),  \chi_q(\pi(D_iD_i^{(-1)})),  \chi_{pq}(\pi(D_iD_i^{(-1)})) \in \mathbb{Z}$.
By Lemma~\ref{LemOrbit0},
\[ \pi(D_i D_i^{(-1)})=t_0+t_{p}P'+t_{q}Q'+t_{pq}(PQ)',\]
for some integers $t_0,t_{p},t_{q},t_{pq}$.
Observe that $ \chi_{pq}(P')= \chi_{pq}(Q')=-1$ and $ \chi_{pq}((PQ)')=1$ as well as
 $ \chi_p(P')=-1,  \chi_p(Q')=q-1$ and $ \chi_p((PQ)')=-(q-1)$. Therefore,
$$ \chi_p(\pi(D_i D_i^{(-1)}))- \chi_{pq}(\pi(D_i D_i^{(-1)}))=q(t_{q}-t_{pq}).$$
It remains to show that there exists $1 \le i \le m$ such that  $ \chi_{pq}(\pi(D_i D_i^{(-1)}))=a_{pq}$ and  $ \chi_p(\pi(D_i D_i^{(-1)}))=a_p$. Clearly, by Results \ref{ResG0} and \ref{ResGn} (c), such $i$ exists. Parts (b) and (c) follow easily from (a) as either $a_p=\lambda$ or $a_{pq}=\lambda$.

We now consider (d). By (b), we have $\lambda>q$. By (a), we conclude that $p|(a_q-a_{pq})$. If $a_q  \neq a_{pq}$, then $\lambda>p$ and therefore, $\lambda>\max(p,q)$. We claim that $a_q=a_{pq}$ leads to $q=2$. Since $ \chi_p \in G_0^*$ and $ \chi_q\in G_N^*$, we conclude from (a) that $q |(\lambda-a_{pq})$ and from Lemma~\ref{LemCharPrime} (b) that
$q |(\lambda + a_{q})$. Consequently, $q|2\lambda$. Since $q \nmid k$, then $q \nmid \lambda$ by Lemma \ref{LemCharPrime} (c). Hence, we derive $q=2$, which is impossible.

Furthermore, by Lemma~\ref{LemCharPrime} (a) and Corollary \ref{CorBasic} (c), $q|(m-2)$ and
\[ \frac{\lambda + a_{pq}}{\gcd (\lambda, a_{pq})} | (m-2). \]
Recall that $q|(\lambda - a_{pq})$ and $q \nmid 2\lambda$. Thus,
\[ q\nmid (\lambda + a_{pq})  \mbox{ and } q\frac{\lambda + a_{pq}}{\gcd (\lambda, a_{pq})}|(m-2).\]
Note that $q\nmid \gcd (\lambda, a_{pq})$ as $q\nmid \lambda$. Therefore,
 $q|\frac{\lambda - a_{pq}}{\gcd (\lambda, a_{pq})}$ and
 \[ \frac{\lambda + a_{pq}}{\gcd (\lambda, a_{pq})} \ge \frac{\lambda - a_{pq}}{\gcd (\lambda, a_{pq})}+2 \geq q+2.\]
Hence, $m\geq q(q+2)+2$.

To show (e), note that by (c), $p| (\lambda-a_q)$ and $p<\lambda$. By Lemma \ref{LemCharPrime} (c), $q \le \frac{\lambda+a_q}{\gcd(\lambda,a_q)}$ and so $\lambda \geq q-1$. Hence, $\lambda \ge \max(p+1, q-1)$.
Recall that there exists an integer $t$ such that
$m-2=t(\frac{\lambda}{a_q}+1)$. Therefore
\[ m=t(\frac{\lambda}{a_q}-1)+2(t+1).\]
As $\gcd(p, k)=1$, $\gcd(p, \lambda)=1$ and $\gcd(p, a_q)=1$. Hence, $p$ divides the integer
$t\frac{\lambda-a_q}{a_q}=t(\frac{\lambda}{a_q}-1)$. Note that $ \chi_p\in G_0^*$ and therefore, $p|m$ and hence $p|(t+1)$. Thus, there exists a positive integer $x$ such that $t=px-1$.
\end{proof}

\medskip

Observe that in (d) above, we require $q\nmid k$ and in (e) above, we require $\gcd(pq,k)=1$.  
In order to apply them, we first show the following lemma describing the relation between $k$ and the prime divisors of $v$.

\begin{lem}\label{LemThreePrimes}
Suppose $v=pqt$ where $p,q$ are distinct primes and $t=1$ or is prime.  
Then $\gcd(pq, k)=1, p$ or $q$. Furthermore, if $\gcd(pq, k)\neq 1$, then $t > \gcd(pq, k)$. 
\end{lem}
\begin{proof} Suppose $\gcd(pq,k)=pq$. Then $p^2q^2|\lambda$. 
As $v=pqt>k>\lambda=p^2q^2$, it is possible only when $t$ is prime with $t\neq p$ and $t\neq q$. 
Clearly, $t\nmid k$ as $v>k$. Thus, by Lemma \ref{LemDivChar1} and \ref{LemCharPrime}, we deduce that  $t|m$ or $t|(m-2)$ . In any case, $m \lambda >pqt$. 

Next, we assume $\gcd(pq, k)=p$. Then $\lambda \geq p^2$.  By the same argument as before, $q|(m-2)$ or $q|m$. Again, we then have $ v=pqt > m\lambda > q p^2.$ 
Therefore, $t>p$. 
\end{proof}

\begin{cor}\label{CorPq}
There exists no $(pq,m,k,\lambda)$-SEDF with $m > 2$.
\end{cor}
\begin{proof}
By Lemma~\ref{LemThreePrimes}, $\gcd(pq,k)=1$. As before, 
we let $ \chi_p, \chi_q$ be characters of order $p$ and $q$ respectively. If both $\chi_p, \chi_q \in G_0^*$ or
both in $G_N^*$, then by Lemma \ref{LemDivChar1} and \ref{LemCharPrime}, $pq|m$ or $pq|(m-2)$. But then $m\geq v$. This is impossible. 

Without loss of generality, we may assume $ \chi_p\in G_0^*$ and $ \chi_q\in G_N^*$. 
Again, we conclude that $p|m$ and $q|(m-2)$. In particular, $m\geq \max(p,q)$. 
Note that by Remark 17, $k\geq 5$. Therefore, $pq > mk\geq 5 \max(p,q)$. 
That means both $p,q$ are odd primes. 
By considering $\chi_{pq}$ and Lemma \ref{LemCharCompare} (d), (e), we see that $k\geq \lambda+1\geq \max(p,q)$. Therefore, 
\[ mk \geq \max(p,q)^2 >pq=v.\]
This is impossible. 
\end{proof}

\begin{rem}
The nonexistence of $(pq,m,k,\lambda)$-SEDF was first recorded in \cite[Theorem 3.9]{BJWZ}. Unfortunately, there is a gap in the proof.
\end{rem}

\begin{thm}\label{3distinct}
Let $p,q,r$ be distinct primes. Then there exists no $(pqr,m,k,\lambda)$-SEDF with $m > 2$.
\end{thm}
\begin{proof} To show $\gcd(pqr,k) \in \{1,p,q,r\}$, we may assume without loss of generality, 
$pq| \gcd(pqr,k)$. But then $pq=\gcd(pq,k)$, which contradicts Lemma \ref{LemThreePrimes}. 
Therefore, $\gcd(pqr,k) \in \{1,p,q,r\}$. Without loss of generality, we may assume $\gcd(pq,k)=1$.

Case (1)  Both $ \chi_p, \chi_q\in G_0^*$ or  $ \chi_p, \chi_q\in G_N^*$. In that case, we would have either 
$pq|m$ or $pq|(m-2)$. In either case, $m\geq pq$. 
As $m\lambda<pqr$, it follows that $\lambda<r$ and $r\nmid \lambda$. 
By Corollary 16, $r>\lambda \geq 4$. We may then assume either $ \chi_p, \chi_q\in G_0^*$ and $ \chi_r\in G_N^*$; or
 $ \chi_p, \chi_q\in G_N^*$ and $ \chi_r\in G_0^*$. Without loss of generality, we may assume
$q>p \ge 2$. We then have either $ \chi_r\in G_0^*$ and $ \chi_q\in G_N^*$; or $ \chi_r\in G_N^*$ and $ \chi_q\in G_0^*$. In either case, we may apply Lemma \ref{LemCharCompare} (d) or (e) to deduce that $k\geq \lambda+1 \geq \max(r, q)$.
Therefore, $mk\ge pqr$.

Case (2)  Without loss of generality, we may assume $ \chi_p\in G_0^*$ and  $ \chi_q\in G_N^*$. 
If $r\nmid k$, then as
$\chi_r\in G_0^*$ or $G_N^*$, we are then back to Case (1).  
Therefore, we may assume $r|k$ and in particular, $r^2|\lambda$. 
Since $m\geq \max(p,q)$, $pqr>mk\geq \max(p,q)r^2$. That means $r< \min(p,q)$. In particular, $p>2$ and $q>2$. If $ \chi_{pq}\in G_N^*$, then by Lemma \ref{LemCharCompare} (d), $mk>m \lambda> q^2 p>pqr$. It remains to deal with the case $ \chi_{pq}\in G_0^*$.
We define $a_q$ as in Lemma \ref{LemCharCompare} and $d=\gcd(\lambda, a_q)$. As $q|(\lambda + a_q)/d$,  by Lemma \ref{LemCharPrime} (c), we have
\[ mk > [(px-1)(\frac{\lambda}{a_q}+1)] \lambda \geq (p-1)(\frac{\lambda}{d}+\frac{a_{q}}{d})\cdot \frac{d\lambda}{a_{q}}\geq (p-1)q \frac{d^2}{a_q} \frac{\lambda}{d}.\]
Recall that $a_q|d^2$ and $d|\lambda$. Hence, $\lambda/d < 2r$. But $r^2|\lambda$, so $r|d$. 
By Lemma \ref{LemCharCompare} (c), $p| d(\lambda/d-a_q/d)$. As $p\nmid \lambda$, 
$p|(\lambda/d -a_q/d)$. Since $\lambda > a_q$, $\lambda \geq pd\geq pr$. Since $m>q$, $m\lambda>pqr$, contradiction. 
\end{proof}

\section{$v=pq^2$}\label{Sec5}

In this section, we show that there is no $(v,m,k,\lambda)$-SEDF with $m>2$ when $v=pq^2$ and $p,q$ are distinct primes. Besides using character theory to study SEDFs, we need to study further restrictions on the parameters by using divisibility argument.

\begin{thm}\label{2distinct}
There exists no $(pq^2,m,k,\lambda)$-SEDF with $m > 2$.
\end{thm}

\begin{proof}
Suppose $\gcd(pq,k)\neq 1$. Then by Lemma \ref{LemThreePrimes}, we must have $\gcd(pq,k)=p$ 
and $p<q$. Note that $p^2 | \lambda$. So, we may write $k=pk'$ and $\lambda=p^2\lambda'$ for some 
integers $k', \lambda'$. We then have $(m-1)k'^2=\lambda'(pq^2-1)$. Let $e:= \gcd(m-1,pq^2-1)$. Then
\[ \frac{(m-1)}{e} k'^2 = \lambda' \frac{(pq^2-1)}{e}.\]
Consequently, $(k')^2\geq  (pq^2-1)/e$. 
Next, we need to find a bound for $e$. 
As either $ \chi_q\in G_0^*$ or $G_N^*$ and $q\nmid k$, hence 
$m = \beta q +2$ or $m = \beta q$ for some integer $\beta$. 
Thus,  
$e= \gcd(\beta q \pm 1, p - \beta^2) \leq |p - \beta^2|$. Therefore, 
$$ mk \geq \beta pq \sqrt{\frac{pq^2-1}{|p - \beta^2|}}.$$
If $p \ge \beta^2$, then $\beta pq \sqrt{\frac{pq^2-1}{p-\beta^2}} \ge \beta pq^2$. If $p < \beta^2$, then $pq \sqrt{\frac{\beta^2pq^2 - \beta^2}{\beta^2-p}} \ge pq\sqrt{q^2} = pq^2$. In any case, $mk \ge pq^2$, a contradiction.

From now on, we assume $\gcd(pq,k)=1$. As shown in the proof of Corollary 27,  $m\geq \max(p,q)$. 
Since $k\geq 5$, $pq^2>mk\geq 5 \max(p,q)$. Therefore, $q>2$. 
We first remove the case $ \chi_q\in G_0^*$. In this case, we have $q|m$ and $q\nmid k$. By Corollary \ref{CorG0}, $q^2|m$. Clearly, $p\nmid m$ and thus $ \chi_p\in G_N^*$. 
By Lemma \ref{LemCharPrime} (c), $p \leq \lambda + 1 \leq k$. Therefore, $mk \geq pq^2$. 

We may now assume $ \chi_q\in G_N^*$. As $q>2$ and $q\nmid k$, $q| (m-2)$.  By Lemma \ref{LemCharPrime} (c), $q \leq \lambda + 1 \leq k$. We claim that $ \chi_p\in G_0^*$. Otherwise, 
 $ \chi_p\in G_N^*$ and $p\mid (m-2)$ also. Then, $mk\geq pq^2$. 
We conclude now that $ \chi_q\in G_N^*$ and $ \chi_p\in G_0^*$.

It is natural to consider the character $ \chi_p \chi_q$. If $ \chi_p \chi_q \in G_N^*$, then as  $q$ is odd, by  Lemma \ref{LemCharCompare} (d), $mk > m \lambda > pq^2$, contradiction. So, we may assume $ \chi_p \chi_q \in G_0^*$.  We claim that $p\neq 2$. 
Otherwise, as $\chi_2\in G_0^*$, $2|m$ and hence $2|(m-2)$. 
But as shown earlier, $q|(m-2)$ and $q\leq k$. Thus, 
\[ mk \geq (m-2) k \geq  2q q=2q^2.\]
This finishes proving our claim. 

To simplify our notation, we denote $a_{ \chi_q}$ by $a_q$.
Let $d:=\gcd(\lambda, a_q)$. 
As $p$ is odd, by Lemma \ref{LemCharCompare} (e), 
we conclude that there exists a positive integer $x$ such that
\[ m=(px-1)\left(\frac{\lambda}{a_q}+1\right)+2 = \frac{(px-1)d}{a_q}\left(\frac{\lambda+a_q}{d}\right)+2 .\]
By Lemma \ref{LemCharPrime} (c), $\lambda/d+a_q/d=qy$ for some integer $y$. 
By Lemma \ref{LemCharCompare} (a), $p|(\lambda -a_q)$. Since $p\nmid k$ and $p\nmid (m-1)$, $p\nmid \lambda$ and $p\nmid d$. Therefore $p|(\lambda/d-a_q/d)$. Write $(\lambda/d-a_q/d)=up$. Note that $u\geq 1$ as
$a_q\neq \lambda$. Since $(a_q/d)|(px-1)$, $\lambda/a_q \geq (x+1)/x$. 
We have thus shown $\lambda/d \geq (x+1)qy/(2x+1)$. 
We then have 
\[ pq^2 > mk >   (px-1)qy\frac{\lambda}{d}\cdot\frac{d^2}{a_q} > (px-1)q^2y^2 \frac{(x+1)d^2}{(2x+1)a_q}.\]
By Result \ref{ResGn} (d),  $d^2/a_q $ is an integer. Note that $(px-1)(x+1)/(2x+1)\geq p$ whenever $x\geq 2$. 
Hence, $x=1$ and so $(x+1)/(2x+1)=2/3$. Thus $d^2/a_q=1$ and $y=1$ as well. 
Note that  $q\nmid d$ as $q\nmid \lambda$ and therefore $(p-1)/d$ is an integer. 
To summarize, we have shown 
that 
\begin{equation}
m = \frac{p-1}{d} q + 2, \  \lambda+a_q=qd \mbox{ where } p>2, \frac{p-1}{d} \in \Z \mbox{ and }a_q=d^2.\end{equation}

\medskip

We may now apply the previous calculation on any character $\psi$ of order $q$, and we conclude that 
$a_{\psi}=d^2$ where $d$ is defined above. 

\medskip

We may also assume  $G=C_p\times C_q\times C_q$. Otherwise, it follows from 
Corollary \ref{LemDivChar3} that $q^2|(m-2)$ and thus $mk > m\lambda > pq^2$. 
Note that $G_N^*$ consists of characters of order $q$. Since $G=C_p\times C_q\times C_q$, there exists $q+1$ characters $ \psi_1, \psi_2, \ldots, \psi_{q+1}$ such that 
\[ G_N^* = \bigcup_{j=1}^{q+1} \{ \psi_j^{t} \ |\ t=1,\ldots, q-1\}.\]
As noted before, $a_{\psi_i}=a_q=d^2$ for all $i$. 

For $1 \leq i \leq m$, define 
\[ A_i:= \{1 \leq j \leq q+1: \ | \psi_j(D_i)|^2 = \frac{\lambda^2}{a_{q}} \}.\]
By applying Fourier inversion formula to compute the coefficient of identity in $D_iD_i^{(-1)}$, we see that
\[ (pq^2)k=k^2+|G_0^*| \lambda + (q-1)(q+1-|A_i|)a_q+ (q-1)|A_i|\frac{\lambda^2}{a_q}.\]
Therefore, for any $i\neq j$, 
\[ 0=(q-1)(|A_j|-|A_i|)(\frac{\lambda^2-a_q^2}{a_q}) . \]
Thus, we see that  $|A_i|=z$ for any $1 \leq i \leq m$. 
Hence, 
\[ |\{ (i,j) : |\psi_{j}(D_i)|^2=\frac{\lambda^2}{a_q}, \ 1\leq i\leq m, \ j=1,\ldots, q+1\}|=mz.\]
On the other hand, by Corollary 8 (c) and Equation (5), for each $1 \le j \le q+1$,
\[ |\{ i : |\psi_{j}(D_i)|^2=\frac{\lambda^2}{a_q}, \ i=1,\ldots, m\}|=(m-2)\frac{a_q}{\lambda+a_q}+1=p.\]
Therefore, $mz=p(q+1)$. Recall that $m=\frac{q(p-1)}{d}+2$. It is easy to see that 
\[ \gcd(m, (q+1))|  (\frac{p-1}{d} -2) . \]
As $m|[p(q+1)]$, we conclude that
$  m $ divides $ p(\frac{p-1}{d}-2)$. We claim that $\frac{p-1}{d}-2\neq 0$. 

Otherwise, in view of Equation (5), we conclude 
$m=2q+2=2(q+1)$. But since $m| p(q+1)$, we must have $p=2$,  which we have ruled out. 
Now, $(\frac{p-1}{d}-2)>0$, we then have 
\[  p(\frac{p-1}{d}-2)\geq m=\frac{q(p-1)}{d}+2. \]
Thus, $p>q$.
This is impossible as we have shown that 
$\lambda/d+a_q/d=q> (\lambda/d-a_q/d)=up\geq p$.  
\end{proof}

\section{$v$ is a prime power}\label{Sec7}

In this section, we consider the case $v=p^n$.
The only known example of nontrivial $(v,m,k,\lambda)$-SEDF is the $(243, 11, 22, 20)$-SEDF in $C_3^5$.
It will then be interesting to find out if there is any other SEDF with $v=p^n$. In this section,
we rule out the case for $n\leq 2$. When $n=3$, we derive strong restrictions on its parameters.
In view of  Theorem \ref{ThmCyclic}, we only deal with noncyclic groups.


\begin{lem}\label{LemPrimePower}
Let $p$ be a prime. Suppose there exists a $(p^n,m,k,\lambda)$-SEDF in $G$ with $m>2$.
\begin{itemize}
\item [(a)] $(p-1)|(k^2-k)$.
\item [(b)] If $p\nmid k$ and $p>2$, then $p|(m-2)$, $(p-1)|k$ and $G_0^*=\emptyset$.
\end{itemize}
\end{lem}

\begin{proof}
By Corollary~\ref{CorDiDi},
\[ D_iD_i^{(-1)}=k+ \sum_{j=1}^{r} a_i\omega(g_j) \]
for some $g_j$'s in $G$ and integers $a_j$'s. Clearly, $|\omega(g_j)|$ is divisible by $p-1$ and therefore,
$(p-1)|(k^2-k)$.

Now, assume that $p\nmid k$ and $p>2$. As $G_N^* \neq \emptyset$, we have $p|(m-2)$ by Corollary \ref{LemDivChar3}. In particular, $p\nmid m$ and hence by Lemma \ref{LemDivChar1}, any character of order $p$-power is in
$G_N^*$ and thus $G_0^*=\emptyset$. 
By Corollary~\ref{CorDiDj}, we conclude $(p-1)|k^2$. Since $p-1$ divides both $k^2$ and $k^2-k$, $(p-1)|k$.
\end{proof}

\medskip

We are now ready to prove the nonexistence of $(p^n,m,k,\lambda)$-SEDF with $m>2$ and $n =2$.

\begin{thm}\label{ThmP2}
Let $p$ be a prime. Then there exists no $(p^2,m,k,\lambda)$-SEDF with $m>2$.
\end{thm}

\begin{proof}
By \cite[Remark 5.17]{JL}, we may assume $p >2$.
If $p|k$, then $p^2|\lambda$ and $k > \lambda \ge p^2$, contradiction. Suppose $p\nmid k$.
Then by Lemma \ref{LemPrimePower} (b), $(p-1)|k$ and $p|(m-2)$. So, $k \ge p-1$ and $m \geq  p+2$. But then $mk>p^2$.
\end{proof}

For the case $v=p^3$, we first show that $G$ is elementary abelian.

\begin{lem}\label{LemPandP2}
Let $p$ be a prime. Then there exists no $(p^3,m,k,\lambda)$-SEDF with $m>2$ in the group $G = C_p \times C_{p^2}$.
\end{lem}

\begin{proof}
By \cite[Remark 5.17]{JL}, we may assume $p>2$. If $p\nmid k$, then by Lemma \ref{LemPrimePower} (b), $p|(m-2)$, $(p-1)|k$ and $G_0^* = \emptyset$.  By Corollary \ref{LemDivChar3}, $p^2|(m-2)$. Then $mk > p^3$, unless $m=p^2 + 2$ and $k=p-1$. But then
$\lambda = [(p-1)^2(p^2+1)]/(p^3-1) \not\in \mathbb{Z}$, contradiction.

Thus, we may assume $p\|k$ and $p^2\|\lambda$. It follows immediately that $m< p$. 
We claim that for any $1\leq i\leq m$, $p|\chi(D_i)$ for any nonprincipal $\chi\in G^*$. 

Since $\chi(D_i)\in \Z[\zeta_{p^2}]$ and prime ideals above $p$ are invariant under complex conjugation in $\Z[\zeta_{p^2}]$, it suffices to show that $p^2| \ |\chi(D_i)|^2$. If $\chi\in G_0^*$, $|\chi(D_i)|^2=\lambda$. So, we need only to consider $\chi\in G_N^*$.

By Corollary~\ref{CorBasic} (c), 
\[  \frac{\lambda}{\gcd(\lambda, a_{\chi})}< \frac{\lambda + a_{\chi}}{\gcd(\lambda, a_{\chi})}\leq m-2 <p. \]
As $p^2|\lambda$, $p^2$ must divide $\gcd(\lambda, a_{\chi})$. Hence $p^2|a_{\chi}$. 
On the other hand, as $v=p^3>\lambda>a_{\chi}$, $p^2 \| a_{\chi}$. 
Therefore, $p^2\| (\lambda^2/a_{\chi})$. We are done by applying Result 5. 

Next, we need to project $G$ to a cyclic group. There exists a subgroup $N$ and a natural projection $\pi:G \to G/N \cong C_{p^2}$. It follows that for any nonprincipal character $\psi$ on $G/N$, $p|\psi(\pi(D_i))$ for any $1\leq i\leq m$. 
By Ma's Lemma \cite[Lemma 1.5.1]{Sch}, for any $1 \leq i \leq m$, we have $\pi(D_i)= pX_i + PY_i$, where $X_i,Y_i \in \mathbb{Z}[G/N]$ and $P$ is the unique subgroup of order $p$ in $G/N$.

Let $\psi$ be a character of order $p$ in $(G/N)^*$.  Note that 
$\psi(\pi(D_i))=p \psi(X_i+Y_i)$ and $p^2\nmid k$. To get our desired result, we show $m\geq p$. 

Case (1) $\psi\in G_0^*$. Then $0=\psi(\pi(D))=\psi (p\sum_{i=1}^m (X_i+Y_i))$. 
In particular, $\psi (\sum_{i=1}^m (X_i+Y_i))=0$.  By Lemma \ref{LemPpower} (b) ,  $p| (mk/p)$.
As $p \nmid (k/p)$, $m\geq p$. 

Case (2) $\psi\in G_N^*$. Then by Lemma~\ref{LemCharBas},  there exist distinct $1 \le i,j \le m$ such that
$\psi(\pi(D-D_i-D_j))=p\psi (\sum_{u \ne i,j}(X_u+Y_u))=0$. By Lemma \ref{LemPpower} (b),  $p| [(m-2)k/p]$. Again, we get  
$p|(m-2)$ and $m>p$. 
\end{proof}

Finally, it remains  to consider $(p^3,m,k,\lambda)$-SEDF in the elementary abelian group $C_p^3$. Although we are not able to settle the nonexistence in all cases, we do eliminate those $p<3\times 10^{12}$. First, we prove the following theorem which greatly restricts the plausible parameter sets of such SEDFs.

\begin{thm}\label{ThmP3}
Let $p$ be a prime. Suppose there exists a $(p^3,m,k,\lambda)$-SEDF with $m>2$ in $G$. Then $G=C_p^3$.
Furthermore,
\begin{enumerate}
\item[(a)] $\lambda = \frac{p^2(p-1)}{3(m-1)}$ and $k = \frac{p(p-1)}{m-1} \sqrt{\frac{p^2+p+1}{3}}$. In particular, $(p-1)\nmid k$, $\frac{p^2+p+1}{3}$ is a square and $p \equiv 1 \pmod{12}$.
\item[(b)] $a=\frac{p^2(p-1)}{3(m-1)(m-3)}$, $a \mid \lambda$ and $3(m-1)(m-3) | (p-1)$.
\end{enumerate}
\end{thm}

\begin{proof}
The fact that $G=C_p^3$ follows from Theorem \ref{ThmCyclic} and Lemma \ref{LemPandP2}. By \cite[Remark 5.17]{JL}, we may assume $p >2$. 

First, we assume $p\nmid k$. By Lemma \ref{LemPrimePower}, $(p-1)|k$ and $p|(m-2)$.
 So, we may write $k = (p-1)k'$ and $m = \alpha p+2$ for some $\alpha, k' \in \mathbb{Z}_{>0}$.
 We may assume $\alpha \leq p$. Otherwise,
\[ mk\geq (p^2+p+2)(p-1)\geq p^3.\]

By Equation \eqref{EqBasic}, we see that $(p-1)|\lambda (p^2+p+1)$. As $\gcd(p-1,p^2+p+1) = \gcd(p-1,3)$, we have $\frac{p-1}{\gcd(p-1,3)} | \lambda$.
Let $w :=\gcd(p-1,3)$.  We may write $\lambda =  \frac{p-1}{w} \lambda'$ for some $\lambda' \in \mathbb{Z}$. Equation \eqref{EqBasic} now becomes \begin{equation}\label{EqBas1} (1 + \alpha p) (k')^2 = \lambda' \frac{1+p+p^2}{w}.\end{equation}
Let $c:= \gcd(1 + \alpha p, \frac{1+p+p^2}{w})$. By applying Euclidean algorithm, we see
that $c|(\alpha^2 - \alpha + 1)$ and we may write $\alpha^2 - \alpha + 1 = cu$ for some $u \in \mathbb{Z}_{>0}$. Equation \eqref{EqBas1} becomes \begin{equation}\label{EqBas2}\frac{1 + \alpha p}{c} (k')^2 = \lambda' \frac{1 + p + p^2}{cw}.\end{equation} And we also have,  \begin{equation}\label{EqS}\lambda' = (\frac{1 + \alpha p}{c} ) \lambda'' \end{equation} for some $\lambda'' \in \mathbb{Z}$.
Note that
\[ m \lambda =\frac{(2+\alpha p)(1 + \alpha p)(p-1)\lambda''}{cw}  .\]
Observe that
\[ (2 + \alpha p)(1 + \alpha p)(p-1)- (\alpha^2 - \alpha + 1)p^3=(\alpha-1)p^2(p-\alpha) + \alpha p (2p-3) + 2(p-1)>0 \]
if $1\leq \alpha \leq p$. As $\alpha \leq p$ and $\alpha^2 - \alpha + 1 = cu$, we thus get
\[ m \lambda  >  \frac{(\alpha^2 - \alpha + 1) p^3 \lambda''}{wc}=  \frac{u\lambda''}{w} p^3.\]

We are done if $w=1$. Thus we may assume $w=3$ and $u\lambda''=1$ or $2$.
Since $u|(\alpha^2 - \alpha + 1)$ so $u=1$. By Equations \eqref{EqBas2} and \eqref{EqS}, we get
\[ (k')^2 = \lambda'' \cdot \frac{1+p+p^2}{3c}. \]
If $\lambda''=2$, then $4|(k')^2$. Consequently, $2| (1+p+p^2)$, which is impossible. Thus, $\lambda''=1$.

Note also that $(p-1)\nmid \lambda$. Otherwise, we may then assume $w$ in Equation \eqref{EqBas1} is $1$ and derive a contradiction using the same argument as before. Hence, we have shown that
\[ \frac{p-1}{3} | \lambda, \ (p-1)\nmid \lambda,  \gcd(1 + \alpha p, \frac{1+p+p^2}{w})=\alpha^2 - \alpha + 1 \mbox{ and } \lambda' = \frac{1 + \alpha p}{c} .\]

We now have
$$\lambda = \frac{(1+\alpha p)}{(\alpha^2 - \alpha +1)} \frac{(p-1)}{3}, \qquad
k = \sqrt{ \frac{1+p+p^2}{3(\alpha^2 - \alpha +1)}} (p-1).$$
Therefore, \begin{equation}\label{EqIneq} \frac{k}{\lambda} = \sqrt{3} \sqrt{ \frac{(1+p+p^2)(\alpha^2 - \alpha +1)}{(1 + \alpha p)^2}}> \frac{3}{2}.\end{equation}
In particular, $k-\lambda > \lambda/2$.

Let $a$ be as defined in Definition 2. Recall that $a=a_{\chi}$ for some $\chi\in G_N^*$. 
Let $d=\gcd(\lambda,a)$ and $x$ be as defined in Theorem \ref{ThmBound}, i.e. $(m-2)=x (\lambda+a)/d$.
It follows from  Lemma \ref{LemCharPrime} (c) and Corollary \ref{CorBasic} (c)  
\[  (\lambda/d+a/d)=p \beta \mbox{ and } (m-2)=p x \beta \]
for some integer $\beta$. In view of Theorem \ref{ThmBound} and $k> 3\lambda/2$, we see that $d/ax >1/2$. That implies 
$x=1$, $\beta=\alpha$ and $d=a$. Hence, $p\alpha=\lambda/a+1$ and $m-2= p\alpha$. 
Note that
\[ 0\equiv \lambda \equiv  (\lambda +a) -a \equiv p\alpha a -a \equiv \alpha a-a \bmod \frac{(p-1)}{3}.\]
If $\alpha =1$, then $(p-1)|\lambda$, which we have ruled out. Therefore, $\alpha a -a= y(p-1)/3$ where
$y\in \Z$. Since $\alpha >1$, $y>0$. Thus, we have 
\[ mk> m\lambda= (p\alpha+2)(p\alpha-1)a > p^2\alpha^2 \frac{y(p-1)}{3(\alpha -1)}.\]
When $\alpha \ge 2$, it is easy to see that the function $\alpha^2/(\alpha-1)$ is increasing as $\alpha$ increases. 
Since $\alpha\geq 2$ and $p\geq 5$ (recall that $3|(1+p+p^2)$), 
\[ mk > p^2 \frac{4(p-1)}{3} >p^3\]
We have thus shown $p|k$.

As $p|k$, we may write $k = p k_1$ and $\lambda = p^2 \lambda_1$ for some $k_1, \lambda_1 \in \mathbb{Z}_{>0}$.
 By Equation \eqref{EqBasic}, we get
\begin{equation}\label{EqBasicP3}  (m-1) k_1^2 = \lambda_1 (p^3-1).\end{equation}
As $(p-1)|(k^2-k)$ by Lemma \ref{LemPrimePower}, it follows that
$(p-1)$ divides $k_1(pk_1-1)$. Let $d_1=\gcd(p-1, k)=\gcd(p-1, k_1)$. Observe that $\gcd(d_1, pk_1-1)=1$.
Write $p-1=d_1 d_2$. As $\gcd(p-1,k)=d_1$, $d_2|(pk_1-1)$. Note that $d_2|(p-1)$. Therefore, $d_2|(k_1-1)$. It follows that $\gcd(d_1, d_2)=1$ and $\gcd(d_2,k_1)=1$.
Thus, $d_2|(m-1)$. Write $k_1=d_1k_1^{\prime}$. In view of Equation \eqref{EqBasicP3}, we get
\begin{equation}\label{EqBas0} \frac{m-1}{d_2} d_1 (k_1^{\prime})^2 = \lambda_1 (p^2+p+1). \end{equation}
Let $z=\gcd(d_1, p^2+p+1)$. Then $(d_1/z)|\lambda_1$ and
$\lambda_1=(d_1/z)t$. As $d_1|(p-1)$, $z=1$ or $3$. Write $m-1=d_2 x$.
We then have
\[ m\lambda \geq (xd_2+1)p^2\frac{td_1}{z}  \geq \frac{tx(p-1)+td_1}{z}p^2.\]
Clearly, $m\lambda\geq p^3$ if $z=1$ or $tx\geq 4$. Thus, $z=3$ and $tx\leq 3$.
As $z=3$, $3|d_1$ and $3|(p-1)$. Moreover, if $tx=3$, then $m\lambda \ge \frac{3(p-1)+3}{3}p^2=p^3$, contradiction. Thus $tx \in \{1,2\}$.

If $x=2$, then $2| [(m-1)/d_2]$. By Equation~\eqref{EqBas0}, we have that $(3\lambda_1)/d_1$ is even and thus, $t$ is even. Hence, $t \ge 2$ and $tx\geq 4$, impossible. If $x=1$ and $t=2$, then $\lambda_1=\frac{2d_1}{3}$ and we rewrite Equation~\eqref{EqBas0} as
$$
(k_1^{\prime})^2 =\frac{2(p^2+p+1)}{3}.
$$
The above equality implies $2 \| (k_1^{\prime})^2$, which is impossible. Thus, $t=1$ and $x=1$. To conclude, we
get
\begin{equation} m-1=d_2,\ \lambda=\frac{p^2 d_1}{3}, k=pd_1\sqrt{\frac{p^2+p+1}{3}}.\end{equation}
As $(p^2+p+1)/3$ is a perfect square, $p^2+p+1 \equiv 3 \bmod 4$. Therefore, $p\equiv 1 \bmod 4$ and $p\equiv 1 \bmod{12}$. Recall that $d_1d_2=p-1$. This proves (a).

For Part (b), note that $\frac{k}{\lambda}=\sqrt{\frac{3(p^2+p+1)}{p^2}}>\sqrt{3}$. By Theorem~\ref{ThmBound}, $1+\frac{\lambda}{ma}>\frac{k}{\lambda}>\sqrt{3}$. Hence, $\frac{\lambda}{a}>(\sqrt{3}-1)m>\frac{m}{2}$. By Corollary~\ref{CorBasic}, $m-2=\mu(\frac{\lambda}{a}+1)$ for some positive integer $\mu$. Hence, $\mu=1$ and $m-3=\frac{\lambda}{a}$. Consequently, $a \mid \lambda$ and $a=\frac{p^2(p-1)}{3(m-1)(m-3)}$. We claim that $\gcd(m-1,p)=\gcd(m-3,p)=1$. Otherwise, $m>p$ and $m\lambda>p^3$, contradiction. Therefore, $3(m-1)(m-3) \mid (p-1)$.
\end{proof}

As a direct application of Theorem~\ref{ThmP3}, the following result suggests that there is no $(p^3,m,k,\lambda)$-SEDF with $m>2$.

\begin{res}\label{p3}
There exists no $(p^3,m,k,\lambda)$-SEDF with $m>2$ and $p$ being a prime less than $ 3 \times 10^{12}$.
\end{res}

\begin{proof}
By Theorem \ref{ThmP3}, $\frac{p^2+p+1}{3}$ is a square. By a computer search, we find that the only primes less than $ 3 \times 10^{12}$ satisfying this property are $p=313$ and $p=2288805793$. In these cases, we have $\frac{p-1}{3} = 2^3 \times 13$ and $\frac{p-1}{3} = 2^5 \times 7 \times 13 \times 37 \times 73 \times 97$ respectively. We may then check that it is not possible to satisfy $(m-1)(m-3)|\frac{p-1}{3}$ unless $p=2288805793$ and $m \in \{5,17,29,149\}$. For these values of $p$ and $m$, letting $k=\frac{p(p-1)}{m-1} \sqrt{\frac{p^2+p+1}{3}}$, we observe that $(p-1) \nmid (k^2-k)$. This contradicts Lemma \ref{LemPrimePower}. Hence, there exists no $(p^3,m,k,\lambda)$-SEDF with $m>2$ when $p$ is a prime less than $ 3 \times 10^{12}$.
\end{proof}

Combining Corollary \ref{CorPq}, Theorems \ref{ThmCyclic}, \ref{3distinct},  \ref{2distinct},  \ref{ThmP2}, \ref{ThmP3}, and Result \ref{p3}, we thus conclude our main result Theorem \ref{ThmMain}.

\section{Concluding Remarks}\label{Sec8}

We have shown that there exists no $(v,m,k,\lambda)$-SEDF in $G$ with $m>2$ and $v$ being a product of at most three not necessarily distinct primes, except possibly when $G=C_p^3$ and $p$ is a prime greater than $3 \times 10^{12}$. When $G=C_p^3$, we derived several restrictions on the parameters. We have done some computer search for prime less than $3\times 10^{12}$ and found that no such SEDF exists.
This strongly suggests that no SEDF exists in $C_p^3$.

We believe the techniques developed in this paper will provide a basic framework for further research on the nonexistence of SEDFs. 
%
%
%
In \cite[Remark 5.17]{JL}, 70 plausible parameter sets for $(v,m,k,\lambda)$-SEDFs with $v \le 10^4$ was listed. By using Result \ref{ResGn}, Corollary \ref{CorBasic} (c), and Inequality~\eqref{EqInequ}, we are able to rule out the 44 cases listed in Table~\ref{Tb1}.
\begin{table}
\caption{Cases eliminated by Result \ref{ResGn}, Corollary \ref{CorBasic}(c), and Inequality~\eqref{EqInequ}}
\label{Tb1}
\begin{center}
\begin{tabular}{|c c c c|} 				   \hline
$v$	&	$m$	&	$k$	& $\lambda$	\\ \hline
784 & 30 & 18 & 12  \\
1089 & 35 & 24 & 18  \\
1540 & 77 & 18 & 16  \\
1701 & 35 & 30 & 18  \\
2401 & 9 & 60 & 12  \\
2401 & 9 & 120 & 48  \\
2401 & 9 & 180 & 108  \\
2401 & 16 & 120 & 90  \\
2401 & 37 & 40 & 24  \\
2401 & 65 & 30 & 24  \\
2500 & 35 & 42 & 24  \\
2500 & 52 & 42 & 36  \\
2625 & 42 & 48 & 36  \\
2784 & 116 & 22 & 20  \\
3025 & 57 & 36 & 24  \\ \hline
\end{tabular}
\begin{tabular}{|c c c c|} 				   \hline
$v$	&	$m$	&	$k$	& $\lambda$	\\ \hline
3381 & 23 & 130 & 110  \\
3888 & 47 & 52 & 32  \\
4096 & 14 & 105 & 35  \\
4096 & 14 & 210 & 140  \\
4225 & 67 & 48 & 36  \\
4375 & 7 & 162 & 36  \\
4375 & 7 & 324 & 144  \\
4375 & 7 & 486 & 324  \\
4375 & 37 & 54 & 24  \\
4375 & 37 & 81 & 54  \\
4564 & 163 & 26 & 24  \\
4625 & 37 & 68 & 36  \\
5376 & 44 & 75 & 45  \\
5776 & 78 & 60 & 48  \\
5832 & 18 & 147 & 63  \\ \hline
\end{tabular}
%
%
%
\begin{tabular}{|c c c c|} 				   \hline
$v$	&	$m$	&	$k$	& $\lambda$	\\ \hline
5832 & 35 & 98 & 56  \\
5832 & 86 & 49 & 35  \\
5888 & 92 & 58 & 52  \\
6400 & 80 & 54 & 36  \\
6656 & 26 & 121 & 55  \\
6860 & 20 & 266 & 196  \\
6860 & 58 & 95 & 75  \\
6976 & 218 & 30 & 28  \\
8281 & 93 & 60 & 40  \\
8625 & 23 & 140 & 50  \\
9801 & 13 & 420 & 216  \\
9801 & 57 & 140 & 112  \\
9801 & 101 & 70 & 50  \\
9801 & 101 & 84 & 72  \\
 & & & \\ \hline
\end{tabular}
\end{center}
\end{table}


In case $v$ is even, there exists $\chi \in G^*$ of order 2. Note that $\chi(D_i)$ is an integer for any $1 \le i \le m$. Hence, $|\chi(D_i)|^2$ must be a square. That means either $\lambda$ or $a_{\chi}$ is a square. If $\lambda$ is not a square, we find all possible values of $a_{\chi}$ and check if $(\lambda+a_\chi)|[\gcd(\lambda,a_\chi)(m-2)]$ holds. 
Employing this idea, we have eliminated $6$ additional cases listed in Table~\ref{Tb2}.

\begin{table}
\caption{Cases eliminated when $v$ is even}
\label{Tb2}
\begin{center}
\begin{tabular}{|c c c c|} 				   \hline
$v$	&	$m$	&	$k$	& $\lambda$	\\ \hline
2376 & 11 & 190 & 152  \\
4096 & 8 & 390 & 260  \\
5832 & 8 & 595 & 425  \\
5832 & 8 & 714 & 612  \\
6656 & 26 & 242 & 220  \\
8960 & 7 & 1054 & 744  \\ \hline
\end{tabular}
\end{center}
\end{table}


As a consequence, there are still 20 plausible parameter sets when $v \le 10^4$, which are tabulated in Table \ref{Tb3}. At this point, we conjecture that the $(243, 11, 22,20)$-SEDF in $C_3^5$ is the only SEDF with $m>2$.

\begin{table}
\caption{Plausible parameter sets for $(v,m,k,\lambda)$-SEDFs with $m>2$ and $v \le 10^4$}
\label{Tb3}
\begin{center}
\begin{tabular}{|c c c c|} 				   \hline
$v$	&	$m$	&	$k$	& $\lambda$	\\ \hline
540 & 12 & 42 & 36  \\
1701 & 35 & 40 & 32  \\
2058 & 86 & 22 & 20  \\
2401 & 7 & 280 & 196  \\
2401 & 9 & 240 & 192  \\
2500 & 18 & 105 & 75  \\
2601 & 53 & 40 & 32  \\ \hline
\end{tabular}
\begin{tabular}{|c c c c|} 				   \hline
$v$	&	$m$	&	$k$	& $\lambda$	\\ \hline
2646 & 16 & 138 & 108  \\
3888 & 24 & 156 & 144  \\
3888 & 47 & 78 & 72  \\
3969 & 32 & 112 & 98  \\
4375 & 7 & 540 & 400  \\
4375 & 9 & 405 & 300  \\
4375 & 16 & 270 & 250  \\ \hline
\end{tabular}
\begin{tabular}{|c c c c|} 				   \hline
$v$	&	$m$	&	$k$	& $\lambda$	\\ \hline
4375 & 37 & 108 & 96  \\
5376 & 44 & 100 & 80  \\
5832 & 18 & 294 & 252  \\
8625 & 23 & 280 & 200  \\
8960 & 32 & 238 & 196  \\
9801 & 26 & 308 & 242  \\
 & & & \\ \hline
\end{tabular}
\end{center}
\end{table}

\end{document}